\documentclass[10pt]{amsart}
\usepackage{amssymb}
\usepackage{bm}
\usepackage{graphicx}
\usepackage[centertags]{amsmath}
\usepackage{amsfonts}
\usepackage{amsthm}
\usepackage{amsbsy}
\usepackage{mathtools}
\usepackage{mathrsfs}
\usepackage[all]{xy}
\newtheorem{thm}{Theorem}[section]
\newtheorem{cor}[thm]{Corollary}
\newtheorem{lem}[thm]{Lemma}

\newtheorem{prop}[thm]{Proposition}

\newtheorem{defn}[thm]{Definition}
\theoremstyle{definition}
\newtheorem{rem}{Remark}
\newtheorem{examp}{Example}

\newcommand{\rr}{\mathbb{R}}
\newcommand{\nn}{\mathbb{N}}
\newcommand{\ee}{\varepsilon}

\newcommand{\meg}{\geqslant}
\newcommand{\mik}{\leqslant}
\newcommand{\ave}{\mathbb{E}}

\newcommand{\bo}{\Omega}

\newcommand{\pp}{\mathbb{P}}
\newcommand{\cala}{\mathcal{A}}
\newcommand{\calb}{\mathcal{B}}
\newcommand{\calc}{\mathcal{C}}
\newcommand{\calf}{\mathcal{F}}
\newcommand{\cals}{\mathcal{S}}
\newcommand{\calp}{\mathcal{P}}
\newcommand{\calq}{\mathcal{Q}}
\newcommand{\calr}{\mathcal{R}}
\newcommand{\calz}{\mathcal{Z}}

\begin{document}

\title{Szemer\'{e}di's regularity lemma via martingales}

\author{Pandelis Dodos, Vassilis Kanellopoulos and Thodoris Karageorgos}

\address{Department of Mathematics, University of Athens, Panepistimiopolis 157 84, Athens, Greece}
\email{pdodos@math.uoa.gr}

\address{National Technical University of Athens, Faculty of Applied Sciences,
Department of Mathematics, Zografou Campus, 157 80, Athens, Greece}
\email{bkanel@math.ntua.gr}

\address{Department of Mathematics, University of Athens, Panepistimiopolis 157 84, Athens, Greece}
\email{tkarageo@math.uoa.gr}


\maketitle


\begin{abstract}
We prove a variant of the abstract probabilistic version of Szemer\'{e}di's regularity lemma, due to Tao, which applies 
to a number of structures (including graphs, hypergraphs, hypercubes, graphons, and many more) and works for random variables
in $L_p$ for any $p>1$. Our approach is based on martingale difference sequences.
\end{abstract}


\section{Introduction}

\numberwithin{equation}{section}

\subsection*{1.1}

The aim of the present paper is to prove a variant of the abstract probabilistic version of Szemer\'{e}di's regularity lemma,
due to Tao \cite{Tao1,Tao2,Tao3}. This variant applies to a number of combinatorial structures---including graphs, hypergraphs, 
hypercubes, graphons, and many more---and works for random variables in $L_p$ for any $p>1$. A proper exposition of our main
result requires some preparatory work and hence, at this point, we will not discuss it in detail. Instead, we will focus
on the following model case which is representative of the contents of this paper.

\subsection*{1.2}

A very basic fact of probability theory is that the set of simple functions is dense in $L_1$. Actually, this fact is so basic
that it is hardly mentioned when applied. But how do we approximate a given random variable by a simple function? More precisely,
given an integrable random variable $f\colon [0,1]\to\rr$ and a real $0<\ee\mik 1$ (that we regard as an error) we are asking for
an effective method to locate a simple function $s\colon [0,1]\to\rr$ such that $\|f-s\|_{L_1}\mik \ee$.

It turns out that there is a natural greedy algorithm for this problem which we are about to describe. We start by setting
$\calf_0=\big\{\emptyset,[0,1]\big\}$ and $f_0=\ave(f\, | \, \calf_0)$. That is, $\calf_0$ is the trivial $\sigma$-algebra on $[0,1]$
and $f_0$ is the conditional expectation of $f$ with respect to $\calf_0$ (see, e.g., \cite{Du}). Notice that $f_0$ is constant
and equal to the expected value $\ave(f)$ of $f$. Thus, if $\|f-f_0\|_{L_1}\mik \ee$, then we are done. Otherwise, by considering 
the support of the positive part or the negative part of $f-f_0$, we may select a measurable subset $A_0$ of $[0,1]$ such that 
\begin{equation} \label{e1.1}
\frac{\ee}{2} < \big| \int_{A_0} (f-f_0)\, dt \big|.
\end{equation}
Next we set $\calf_1=\sigma(\calf_0 \cup\{A_0\})$ and  $f_1=\ave(f\, | \, \calf_1)$. (That is, $\calf_1$ is the smallest
$\sigma\text{-algebra}$ on $[0,1]$ that contains all elements of $\calf_0$ and $A_0$, and $f_1$ is the conditional expectation
of $f$ with respect to $\calf_1$.) Observe that, by \eqref{e1.1}, we have
\begin{equation} \label{e1.2}
\frac{\ee}{2} < \big| \int_{A_0} (f_1-f_0)\, dt \big| \mik \|f_1-f_0\|_{L_1}.
\end{equation}
Also notice that $f_1$ is a simple function since the $\sigma$-algebra $\calf_1$ is finite, and so if $\|f-f_1\|_{L_1}\mik \ee$, 
then we can stop this process. On the other hand, if $\|f-f_1\|_{L_1}>\ee$, then we select a measurable subset $A_1$ of $[0,1]$ 
such that $|\int_{A_1} (f-f_1)\, dt|> \ee/2$ and we continue similarly.

The next thing that one is led to analyze is whether this algorithm will eventually terminate and, if yes, at what speed.
To this end, notice that if the algorithm runs forever, then it produces an increasing sequence $(\calf_i)$ of finite 
$\sigma$-algebras of $[0,1]$ and a sequence $(f_i)$ of random variables with $f_i=\ave(f\, | \, \calf_i)$ for every
$i\in\nn$ and such that $\|f_i-f_{i-1}\|_{L_1}> \ee/2$\, if\, $i\meg 1$. In other words, $(f_i)$ is a martingale adapted to the
filtration $(\calf_i)$ whose successive differences are bounded away from zero in the $L_1$ norm. This last piece of information
is the key observation of this analysis since successive differences of martingales, known as martingale difference sequences,
are highly structured sequences of random variables. In particular, if the given random variable $f$ belongs to $L_p$ for some
$1< p\mik 2$, then for every integer $n\meg 1$ we have
\begin{equation} \label{e1.3}
\Big( \sum_{i=1}^n \|f_i-f_{i-1}\|^2_{L_p}\Big)^{1/2} \mik \Big( \frac{1}{p-1}\Big)^{1/2} \cdot \|f\|_{L_p}.
\end{equation}
This functional analytic estimate is sharp, and was recently proved by Ricard and Xu \cite{RX} who deduced it from 
a uniform convexity inequality for $L_p$ spaces. We briefly comment on these results in Appendix A. 

Of course, with inequality \eqref{e1.3} at our disposal, it is very easy to analyze the greedy algorithm described above.
Precisely,  by \eqref{e1.3} and the monotonicity of the $L_p$ norms, we see that if $f\in L_p$ for some $1<p \mik 2$, then 
this algorithm will terminate after at most $\lfloor 4\, \|f\|^2_{L_p} \ee^{-2} (p-1)^{-1}\rfloor+1$ iterations.

\subsection*{1.3}

Our main result (Theorem \ref{t3.1} in Section 3) follows the method outlined above, but with two important extra features.

First, our approximation scheme is more demanding in the sense that the simple function we wish to locate is required
to be a linear combination of characteristic functions of sets belonging to a given class. It is useful to view the sets 
in this class as being ``structured"\!, though for the purpose of performing the greedy algorithm only some (not particularly
restrictive) stability properties are needed. These properties are presented in Definition \ref{d2.1} in Section 2, together
with several related examples. 

Second, the error term of the approximation is controlled not only by the $L_p$ norm but also by a certain ``uniformity norm" 
which depends on the class of ``structured" sets with which we are dealing (see Definition \ref{d2.2} in Section 2). 
This particular feature is already present in Tao's work and can be traced to \cite{FrK}. 

Finally, we note that in Section 4 we discuss some applications, including a regularity lemma for hypercubes and
an extension of the strong regularity lemma to $L_p$ graphons for any $p>1$. More applications will appear in \cite{DKK}.

\subsection*{1.4}

By $\nn=\{0,1,2,\dots\}$ we denote the set of natural numbers. As usual, for every positive integer $n$ we set
$[n]\coloneqq \{1,\dots,n\}$. For every function $f\colon\nn\to\nn$ and every $\ell\in\nn$ by $f^{(\ell)}\colon\nn\to\nn$ 
we shall denote the $\ell$-th iteration of $f$ defined recursively by $f^{(0)}(n)=n$ and $f^{(\ell+1)}(n)=f\big(f^{(\ell)}(n)\big)$
for every $n\in\nn$. All other pieces of notation we use are standard.


\section{Semirings and their uniformity norms}

\numberwithin{equation}{section}

We begin by introducing the following slight strengthening of the classical concept of a semiring of sets (see also \cite{BN}).
\begin{defn} \label{d2.1}
Let $\Omega$ be a nonempty set and $k$ a positive integer. Also let $\cals$ be a collection of subsets of $\Omega$. 
We say that $\mathcal{S}$ is a \emph{$k$-semiring on $\Omega$} if the following properties are satisfied.
\begin{enumerate}
\item[(P1)] We have that\, $\emptyset,\Omega\in\cals$.
\item[(P2)] For every $S,T\in\cals$ we have that $S\cap T\in\cals$.
\item[(P3)] For every $S,T\in\cals$ there exist $\ell\in [k]$ and $R_1,\dots,R_{\ell}\in \cals$ which are pairwise disjoint 
and such that $S\setminus T= R_1\cup\cdots\cup R_{\ell}$.
\end{enumerate}
\end{defn}
As we have already indicated in the introduction, we view every element of a $k\text{-semiring}$ $\cals$ as a ``structured" set and
a linear combination of few characteristic functions of elements of $\cals$ as a ``simple" function. We will use the following norm 
in order to quantify how far from being ``simple" a given function is.
\begin{defn} \label{d2.2}
Let $(\Omega,\calf,\mathbb{P})$ be a probability space, $k$ a positive integer and $\cals$ a $k$-semiring on $\Omega$ with 
$\cals\subseteq \calf$. For every $f\in L_1(\Omega,\calf,\mathbb{P})$ we set
\begin{equation} \label{e2.1}
\|f\|_{\cals} = \sup\Big\{ \big| \int_S f \, d\mathbb{P} \big| : S\in \cals\Big\}.
\end{equation}
The quantity $\|f\|_{\cals}$ will be called the \emph{$\cals$-uniformity norm} of $f$.
\end{defn}
The $\cals$-uniformity norm is, in general, a seminorm. Note, however, that if the $k$-semiring $\cals$ is sufficiently rich, then 
the function $\|\cdot\|_{\cals}$ is indeed a norm. More precisely, the function $\|\cdot\|_{\cals}$ is a norm if and only if the family 
$\{\mathbf{1}_S:S\in \cals\}$ separates points in $L_1(\Omega,\calf,\mathbb{P})$, that is, for every $f,g\in L_1(\Omega,\calf,\mathbb{P})$
with $f\neq g$ there exists $S\in\cals$ with $\int_S f\, d\mathbb{P}\neq \int_S g\, d\mathbb{P}$.

The simplest example of a $k$-semiring on a nonempty set $\Omega$, is an algebra of subsets of $\Omega$. Indeed, observe that a family
of subsets of $\Omega$ is a $1$-semiring if and only if it is an algebra. Another basic example is the collection of all intervals of
a linearly ordered set, a family which is easily seen to be a $2$-semiring. More interesting (and useful) $k$-semirings can be constructed
with the following lemma.
\begin{lem} \label{l2.3}
Let $\Omega$ be a nonempty set. Also let $m, k_1,\dots, k_m$ be positive integers and set $k=\sum_{i=1}^m k_i$.
If\, $\cals_i$ is a $k_i$-semiring on $\Omega$ for every $i\in [m]$, then the family
\begin{equation} \label{e2.2}
\cals=\Big\{\bigcap_{i=1}^m S_i: S_i\in\cals_i \text{ for every } i\in [m] \Big\}
\end{equation}
is a $k$-semiring on $\Omega$.
\end{lem}
\begin{proof}
Clearly we may assume that $m\meg 2$. Notice, first, that the family $\cals$ satisfies properties (P1) and (P2) in Definition \ref{d2.1}.
To see that property (P3) is also satisfied, fix $S, T\in\cals$ and write $S =\bigcap_{i=1}^m S_i$ and $T=\bigcap_{i=1}^m T_i$ where 
$S_i, T_i\in\cals_i$ for every $i\in [m]$. We set $P_1=\Omega\setminus T_1$ and $P_j=T_1\cap\cdots\cap T_{j-1}\cap (\Omega\setminus T_j)$
if $j\in\{2,\dots,m\}$. Observe that the sets $P_1,\dots, P_m$ are pairwise disjoint. Moreover,
\begin{equation} \label{e2.3}
\Omega\setminus \Big(\bigcap_{i=1}^m T_i \Big)  =  \bigcup_{j=1}^m P_j 
\end{equation}
and so
\begin{equation} \label{e2.4}
S\setminus T = \Big( \bigcap_{i=1}^m S_i\Big) \setminus \Big( \bigcap_{i=1}^m T_i\Big) =
\bigcup_{j=1}^m\Big(\bigcap_{i=1}^m S_i \cap P_j\Big). 
\end{equation}
Let $j\in [m]$ be arbitrary. Since $\mathcal{S}_j$ is a $k_j$-semiring, there exist $\ell_j\in [k_j]$ and pairwise disjoint sets
$R^j_1,\dots, R^j_{\ell_j}\in\cals_j$ such that $S_j\setminus T_j=R^j_1\cup \cdots \cup R^j_{\ell_j}$. Thus, setting
\begin{enumerate}
\item[(a)] $B_1=\Omega$ and $B_j=\bigcap_{1\mik i<j}(S_i\cap T_i)$ if $j\in\{2,\dots,m\}$,
\item[(b)] $C_j=\bigcap_{j<i\mik m} S_i$ if $j\in \{1,\dots,m-1\}$ and $C_m=\Omega$,
\end{enumerate}
and invoking the definition of the sets $P_1,\dots,P_m$ we obtain that
\begin{equation} \label{e2.5}
S\setminus T = \bigcup_{j=1}^m \Big( \bigcup_{n=1}^{\ell_j} \big( B_j \cap R^j_n\cap C_j \big) \Big).
\end{equation}
Now set $I=\bigcup_{j=1}^{m} \big(\{j\}\times [\ell_j]\big)$ and observe that $|I|\mik k$. For every $(j,n)\in I$ let $U^j_n=B_j \cap R^j_n\cap C_j$
and notice that $U^j_n\in\cals$, $U_n^j\subseteq R_n^j$ and $U_n^j\subseteq P_j$. It follows that the family $\{U^j_n: (j,n)\in I\}$ is contained
in $\cals$ and consists of pairwise disjoint sets. Moreover, by \eqref{e2.5}, we have
\begin{equation} \label{e2.6}
S\setminus T= \bigcup_{(j,n)\in I} U_n^j.
\end{equation}
Hence, the family $\cals$ satisfies property (P3) in Definition \ref{d2.1}, as desired.
\end{proof}
By Lemma \ref{l2.3}, we have the following corollary.
\begin{cor} \label{c2.4}
The following hold.
\begin{enumerate}
\item[(a)] Let $\Omega$ be a nonempty set. Also let $k$ be a positive integer and for every $i\in [k]$ let $\cala_i$ be an algebra on $\Omega$.
Then the family
\begin{equation} \label{e2.7}
\{A_1\cap \cdots \cap A_k: A_i\in \cala_i \text{ for every } i\in [k]\}
\end{equation}
is a $k$-semiring on $\Omega$.
\item[(b)] Let $d,k_1,\dots,k_d$ be a positive integers and set $k=\sum_{i=1}^d k_i$. Also let $\Omega_1,\dots,\Omega_d$ be nonempty sets and
for every $i\in [d]$ let $\cals_i$ be a $k_i$-semiring on $\Omega_i$. Then the family
\begin{equation} \label{e2.8}
\{S_1\times \cdots\times S_d: S_i\in \cals_i \text{ for every } i\in [d] \}
\end{equation}
is $k$-semiring on $\Omega_1\times\cdots\times \Omega_d$.
\end{enumerate}
\end{cor}
Next we isolate some basic properties of the $\cals$-uniformity norm.
\begin{lem} \label{l2.5}
Let $(\Omega,\calf,\mathbb{P})$ be a probability space, $k$ a positive integer and $\cals$ a $k\text{-semiring}$ on $\Omega$
with $\cals\subseteq \calf$. Also let $f\in L_1(\Omega,\calf,\mathbb{P})$. Then the following hold.
\begin{enumerate}
\item[(a)] We have $\|f\|_{\cals}\mik \|f\|_{L_1}$.
\item[(b)] If $\mathcal{B}$ is a $\sigma$-algebra on $\Omega$ with $\mathcal{B}\subseteq \cals$, 
then $\|\ave(f\, | \, \mathcal{B})\|_{\cals}\mik \|f\|_{\cals}$.
\item[(c)] If $\cals$ is a $\sigma$-algebra, then $\|f\|_{\cals} \mik \|\ave(f\, | \, \cals)\|_{L_1}\mik 2 \|f\|_{\cals}$.
\end{enumerate}
\end{lem}
\begin{proof}
Part (a) is straightforward. For part (b), fix a $\sigma$-algebra $\mathcal{B}$ on $\Omega$ with $\mathcal{B}\subseteq \cals$
and set $P=\{\omega\in\Omega: \ave(f\, | \, \mathcal{B})(\omega)\meg 0\}$ and $N=\Omega\setminus P$. Notice that 
$P,N\in\mathcal{B}\subseteq\cals$. Hence, for every $S\in\cals$ we have 
\begin{eqnarray} \label{e2.9}
\big|\int_S \ave(f\, | \, \mathcal{B}) \, d\mathbb{P}\big| & \mik & \max\Big\{ \int_{P\cap S} \ave(f\, | \, \mathcal{B})\, d\mathbb{P},
-\int_{N\cap S} \ave(f\, | \, \mathcal{B})\, d\mathbb{P}\Big\} \\
& \mik & \max\Big\{ \int_{P} \ave(f\, | \, \mathcal{B})\, d\mathbb{P}, -\int_{N} \ave(f\, | \, \mathcal{B})\, d\mathbb{P}\Big\} \nonumber \\
& = & \max\Big\{ \int_{P} f\, d\mathbb{P}, -\int_{N} f\, d\mathbb{P}\Big\} \mik \|f\|_{\cals} \nonumber
\end{eqnarray}
which yields that $\|\ave(f\, | \, \mathcal{B})\|_{\cals}\mik \|f\|_{\cals}$. 

Finally, assume that $\cals$ is a $\sigma$-algebra and notice that $\int_S f\, d\pp=\int_S \ave(f\, | \, \cals)\, d\pp$ 
for every $S\in\cals$. In particular, we have $\|f\|_{\cals}\mik \|\ave(f\, |\, \cals)\|_{L_1}$. Also let, as above, 
$P=\{\omega\in\Omega: \ave(f\, | \, \cals)(\omega)\meg 0\}$ and $N=\Omega\setminus P$. Since $P,N\in\cals$ we obtain that
\begin{equation} \label{e2.10}
\|\ave(f\, | \, \cals)\|_{L_1} \mik 2\cdot \max\Big\{ \int_P \ave(f\, | \, \cals)\, d\mathbb{P},
-\int_N \ave(f\, | \, \cals)\, d\mathbb{P}\Big\} \mik 2\|f\|_{\cals}
\end{equation}
and the proof is completed.
\end{proof}
We close this section by presenting some examples of $k$-semirings which are relevant from a combinatorial perspective. 
In the first example the underlying space is the Cartesian product of a finite sequence of nonempty finite sets. 
The corresponding semirings are related to the development of Szemer\'{e}di's regularity method for hypergraphs.
\begin{examp} \label{ex1}
Let $d\in\nn$ with $d\meg 2$ and $V_1,\dots,V_d$ nonempty finite sets. We view the Cartesian product $V_1\times\cdots \times V_d$
as a discrete probability space equipped with the uniform probability measure. For every nonempty subset $F$ of $[d]$ let
$\pi_F\colon\prod_{i\in [d]} V_i\to \prod_{i\in F} V_i$ be the natural projection and set
\begin{equation} \label{e2.11}
\cala_F=\Big\{ \pi^{-1}_F(A): A\subseteq \prod_{i\in F} V_i \Big\}.
\end{equation}
The family $\cala_F$ is an algebra of subsets of $V_1\times\cdots \times V_d$ and consists of those sets which
depend only on the coordinates determined by $F$. 

More generally, let $\calf$ be a family of nonempty subsets of $[d]$. Set $k=|\calf|$ and observe that, by Corollary \ref{c2.4}, 
we may associate with the family $\calf$ a $k$-semiring $\cals_{\calf}$ on $V_1\times\cdots\times V_d$ defined by the rule
\begin{equation} \label{e2.12}
S\in\cals_{\calf} \Leftrightarrow S=\bigcap_{F\in\calf} A_F \text{ where } A_F\in \cala_F \text{ for every } F\in\calf.
\end{equation}
Notice that if the family $\calf$ satisfies  $[d]\notin \calf$ and $\cup\calf=[d]$, then it gives rise to a non-trivial semiring
whose corresponding uniformity norm is a genuine norm.

It turns out that there is a minimal non-trivial semiring $\cals_{\min}$ one can obtain in this way. It corresponds to the family
$\calf_{\min}={[d]\choose 1}$ and is particularly easy to grasp since it consists of all rectangles of $V_1\times\cdots\times V_d$.
The $\cals_{\min}$-uniformity norm is known as the \emph{cut norm} and was introduced by Frieze and Kannan \cite{FrK}.

At the other extreme, this construction also yields a maximal non-trivial semiring $\cals_{\max}$ on $V_1\times\cdots\times V_d$.
It corresponds to the family $\calf_{\max}={[d]\choose d-1}$ and consists of those subsets of the product which can be written as
$A_1\cap\cdots \cap  A_d$ where for every $i\in [d]$ the set $A_i$ does not depend on the $i$-th coordinate. The $\cals_{\max}$-uniformity
norm is known as the \emph{Gowers box norm} and was introduced by Gowers \cite{Go1,Go2}.
\end{examp}
In the second example the underlying space is of the form $\bo\times\bo$ where $\bo$ is the sample space of a probability space
$(\bo,\calf,\pp)$. The corresponding semirings are related to the theory of convergence of graphs (see, e.g., \cite{BCLSV,L}).
\begin{examp} \label{ex2}
Let $(\bo,\calf,\pp)$ be a probability space and define
\begin{equation} \label{e2.13}
\cals_{\square}=\big\{ S\times T: S,T\in\calf\big\}.
\end{equation}
That is, $\cals_{\square}$ is the family of all measurable rectangles of $\bo\times\bo$. By Corollary \ref{c2.4}, we see that
$\cals_{\square}$ is a $2$-semiring on $\bo\times\bo$. The $\cals_{\square}$-uniformity norm is also referred to as the
\textit{cut norm} and is usually denoted by $\|\cdot\|_{\square}$. In particular, for every integrable random variable
$f\colon \bo\times\bo\to\rr$ we have
\begin{equation} \label{e2.14}
\|f\|_{\square} = \sup\Big\{ \big| \int_{S\times T} f\, d\pp\big|: S,T\in\calf \Big\}.
\end{equation}
There is another natural semiring in this context which was introduced by Bollob\'{a}s and Nikiforov \cite{BN} and can 
be considered as the ``symmetric" version of $\cals_{\square}$. Specifically, let
\begin{equation} \label{e2.15}
\Sigma_{\square}=\big\{ S\times T: S,T\in\calf \text{ and either } S=T \text{ or } S\cap T=\emptyset\big\}
\end{equation} 
and observe that $\Sigma_{\square}$ is a $4$-semiring which is contained, of course, in $\cals_{\square}$.
On the other hand, note that the family $\cals_{\square}$ is not much larger than $\Sigma_{\square}$ since every element of 
$\cals_{\square}$ can be written as the disjoint union of at most $4$ elements of $\Sigma_{\square}$. Therefore, 
for every integrable random variable $f\colon\bo\times\bo\to\rr$ we have
\begin{equation} \label{e2.16}
\|f\|_{\Sigma_{\square}} \mik \|f\|_{\square} \mik 4 \|f\|_{\Sigma_{\square}}.
\end{equation}
\end{examp}
In the last example the underlying space is the hypercube
\begin{equation} \label{e2.17}
A^n=\big\{ (a_0,\dots,a_{n-1}): a_0,\dots,a_{n-1}\in A\big\}
\end{equation}
where $n$ is a positive integer and $A$ is a finite alphabet (i.e., a finite set) with at least two letters. 
The building blocks of the corresponding semirings were introduced by Shelah \cite{Sh} in his work on the Hales--Jewett 
numbers, and are essential tools in all known combinatorial proofs of the density Hales--Jewett theorem (see \cite{DKT1,P,Tao3}).
\begin{examp} \label{ex3}
Let $n$ be a positive integer and $A$ a finite alphabet with $|A|\meg 2$. As in Example \ref{ex1}, we view the hypercube $A^n$ 
as a discrete probability space equipped with the uniform probability measure. 

Now let $a,b\in A$ with $a\neq b$. Also let $z, y\in A^n$ and write $z=(z_0,\dots,z_{n-1})$ and $y=(y_0,\dots,y_{n-1})$.
We say that $z$ and $y$ are \textit{$(a,b)$-equivalent} provided that for every $i\in\{0,\dots,n-1\}$
and every $\gamma\in A\setminus\{a,b\}$ we have
\begin{equation} \label{e2.18}
z_i=\gamma \ \text{ if and only if } \ y_i=\gamma.
\end{equation}
In other words, $z$ and $y$ are $(a,b)$-equivalent if they possibly differ only in the coordinates taking values in $\{a,b\}$.
Clearly, the notion of $(a,b)$-equivalence defines an equivalence relation on $A^n$. The sets which are invariant under this equivalence 
relation are called \textit{$(a,b)$-insensitive}. That is, a subset $X$ of $A^n$ is $(a,b)$-insensitive provided that for every $z\in X$
and every $y\in A^n$ if $z$ and $y$ are $(a,b)$-equivalent, then $y\in X$. We set
\begin{equation} \label{e2.19}
\cala_{\{a,b\}} = \{ X\subseteq A^n:  X \text{ is $(a,b)$-insensitive} \}.
\end{equation}
It follows readily from the above definitions that the family $\cala_{\{a,b\}}$ is an algebra of subsets of $A^n$. 

The algebras $\big\{\cala_{\{a,b\}}: \{a,b\}\in {A\choose 2}\big\}$ can then be used to construct various $k\text{-semirings}$
on $A^n$. Specifically, let $\calf\subseteq {A\choose 2}$ and set $k=|\calf|$. By Corollary \ref{c2.4}, we see that the family 
constructed from the algebras $\{\cala_{\{a,b\}}: \{a,b\}\in\calf\}$ via formula \eqref{e2.7} is a $k$-semiring on $A^n$.

The maximal semiring obtained in this way corresponds to the family ${A\choose 2}$. We shall denote it by $\cals(A^n)$. In
particular, we have that $\cals(A^n)$ is a $K$-semiring on $A^n$ where $K=|A|(|A|-1)2^{-1}$. Note that $K$ is independent of $n$.
Also observe that if $|A|\meg 3$, then the $\cals(A^n)$-uniformity norm is actually a norm.
\end{examp}


\section{The main result}

\numberwithin{equation}{section}

First we introduce some terminology and some pieces of notation. We say that a function $F\colon\nn\to\rr$ is a
\textit{growth function} provided that: (i) $F$ is increasing, and (ii)~$F(n)\meg n+1$ for every $n\in\nn$. Moreover, 
for every nonempty set $\Omega$ and every finite partition $\calp$ of $\Omega$ by $\cala_{\calp}$ we shall denote the
$\sigma\text{-algebra}$ on $\Omega$ generated~by~$\calp$. Clearly, the $\sigma$-algebra $\cala_{\calp}$ is finite and
its nonempty atoms are precisely the members of $\calp$. Also note if $\calq$ and $\calp$ are two finite partitions of 
$\Omega$, then $\calq$ is a refinement of $\calp$ if and only if $\cala_{\calq}\supseteq \cala_{\calp}$.

Now for every pair $k,\ell$ of positive integers, every $0<\sigma\mik 1$, every $1<p\mik 2$ and every growth function 
$F\colon\nn\to\rr$ we define $h\colon \nn\to\nn$ recursively by the rule
\begin{equation} \label{e3.1}
\begin{cases}
h(0)=0, \\
h(i+1)=h(i)+\lceil \sigma^2\, \ell\, F^{(h(i)+2)}\!(0)^2 (p-1)^{-1}\rceil
\end{cases}
\end{equation}
and we set
\begin{equation} \label{e3.2}
R= h\big(\lceil \ell\, \sigma^{-2}(p-1)^{-1}\rceil -1 \big).
\end{equation}
Finally, we define 
\begin{equation} \label{e3.3}
\mathrm{Reg}(k,\ell,\sigma,p,F)= F^{(R)}(0).
\end{equation}
Note that if $F\colon\nn\to\nn$ is a primitive recursive growth function which belongs to the class $\mathcal{E}^n$
of Grzegorczyk's hierarchy for some $n\in\nn$ (see, e.g., \cite{Ros}), then the numbers $\mathrm{Reg}(k,\ell,\sigma,p,F)$ 
are controlled by a primitive recursive function belonging to the class $\mathcal{E}^m$ where $m=\max\{4, n+2\}$.

We are now ready to state the main result of this paper. 
\begin{thm} \label{t3.1}
Let $k, \ell$ be positive integers, $0<\sigma\mik 1$, $1<p\mik 2$ and $F\colon\nn\to \rr$ a growth function. Also let 
$(\bo,\calf,\pp)$ be a probability space and $(\cals_i)$ an increasing sequence of $k$-semirings on $\Omega$
with $\cals_i\subseteq \calf$ for every $i\in\nn$. Finally, let $\calc$ be a family in $L_p(\bo,\calf,\pp)$ such that 
$\|f\|_{L_p}\mik 1$ for every $f\in\calc$ and with $|\calc|=\ell$. Then there exist 
\begin{enumerate}
\item[(a)] a natural number $N$ with $N\mik \mathrm{Reg}(k,\ell,\sigma,p,F)$,
\item[(b)] a partition $\calp$ of $\Omega$ with $\calp\subseteq \cals_N$ and $|\calp|\mik (k+1)^N$, and
\item[(c)] a finite refinement $\calq$ of $\calp$ with $\calq\subseteq\cals_i$ for some $i\meg N$
\end{enumerate}
such that for every $f\in\calc$, writing $f=f_{\mathrm{str}}+ f_{\mathrm{err}}+ f_{\mathrm{unf}}$ where
\begin{equation} \label{e3.4}
f_{\mathrm{str}}=\ave(f \, | \, \cala_{\calp}), \ \
f_{\mathrm{err}}=\ave(f \, | \, \cala_{\calq})-\ave(f \, | \, \cala_{\calp}) \ \text{ and } \
f_{\mathrm{unf}}=f-\ave(f \, | \, \cala_{\calq}), 
\end{equation}
we have the estimates 
\begin{equation} \label{e3.5}
\| f_{\mathrm{err}}\|_{L_p}\mik \sigma \ \text{ and } \ \|f_{\mathrm{unf}}\|_{\cals_i}\mik \frac{1}{F(i)}
\end{equation}
for every $i\in\{0,\dots,F(N)\}$.
\end{thm}
The case ``$p=2$" in Theorem \ref{t3.1} is essentially due to Tao \cite{Tao1,Tao2,Tao3}. His approach, however,
is somewhat different since he works with $\sigma\text{-algebras}$ instead of $k$-semirings.

The increasing sequence $(\cals_i)$ of $k$-semirings can be thought of as the higher-complexity analogue of the 
classical concept of a filtration in the theory of martingales. In fact, this is more than an analogy since, by 
applying Theorem \ref{t3.1} to appropriately selected filtrations, one is able to recover the fact that, for any 
$1<p\mik 2$, every $L_p$ bounded martingale is $L_p$ convergent. We discuss these issues in Appendix B.

We also note that the idea to obtain ``uniformity" estimates with respect to an arbitrary growth function has been 
considered by several authors. This particular feature is essential when one wishes to iterate this structural decomposition 
(this is the case, for instance, in the context of hypergraphs---see, e.g., \cite{Tao1}). On the other hand, the need
to ``regularize"\!, simultaneously, a finite family of random variables appears frequently in extremal combinatorics and 
related parts of Ramsey theory (see, e.g., \cite{DKT2}). Nevertheless, in most applications (including the applications 
presented in Section 4), one deals with a single random variable and with a single semiring. Hence, we will isolate
this special case in order to facilitate future references.

To this end, for every positive integer $k$, every $0<\sigma\mik1$, every $1<p\mik 2$ and every growth function
$F\colon\nn\to\rr$ we set
\begin{equation} \label{e3.6}
\mathrm{Reg}'(k,\sigma,p,F)=(k+1)^{\mathrm{Reg}(k,1,\sigma,p,F')}
\end{equation}
where $F'\colon \nn\to\rr$ is the growth function defined by the rule $F'(n)=F\big((k+1)^n\big)$ for every $n\in\nn$.
We have the following corollary.
\begin{cor} \label{c3.2}
Let $k$ be a positive integer, $0<\sigma\mik 1$, $1<p\mik 2$ and $F\colon\nn\to \rr$ a growth function. Also let 
$(\bo,\calf,\pp)$ be a probability space and let $\cals$ be a $k\text{-semiring}$ on $\Omega$ with $\cals\subseteq\calf$. 
Finally, let $f\in L_p(\bo,\calf,\pp)$ with $\|f\|_{L_p}\mik 1$. Then there exist
\begin{enumerate}
\item[(a)] a positive integer $M$ with $M\mik \mathrm{Reg}'(k,\sigma,p,F)$,
\item[(b)] a partition $\calp$ of $\Omega$ with $\calp\subseteq \cals$ and $|\calp|= M$, and 
\item[(c)] a finite refinement $\calq$ of $\calp$ with $\calq\subseteq\cals$
\end{enumerate}
such that, writing $f=f_{\mathrm{str}}+ f_{\mathrm{err}}+ f_{\mathrm{unf}}$ where
\begin{equation} \label{e3.7}
f_{\mathrm{str}}=\ave(f \, | \, \cala_{\calp}), \ \
f_{\mathrm{err}}=\ave(f \, | \, \cala_{\calq})-\ave(f \, | \, \cala_{\calp}) \ \text{ and } \
f_{\mathrm{unf}}=f-\ave(f \, | \, \cala_{\calq}), 
\end{equation}
we have the estimates
\begin{equation} \label{e3.8}
\| f_{\mathrm{err}}\|_{L_p}\mik \sigma \ \text{ and } \ \|f_{\mathrm{unf}}\|_{\cals}\mik \frac{1}{F(M)}.
\end{equation}
\end{cor}
Finally, we notice that the assumption that $1< p\mik 2$ in the above results is not restrictive, since the case of random 
variables in $L_p$ for $p>2$ is reduced to the case $p=2$. On the other hand, we remark that Theorem \ref{t3.1} does
not hold true for $p=1$ (see Appendix B). Thus, the range of $p$ in Theorem \ref{t3.1} is optimal.


\subsection{Proof of Theorem \ref{t3.1}}

We start with the following lemma.
\begin{lem} \label{l3.3}
Let $k$ be a positive integer, $p\meg 1$ and $0<\delta\mik 1$. Also let $(\bo,\calf,\pp)$ be a probability
space, $\Sigma$ a $k$-semiring on $\Omega$ with $\Sigma\subseteq \calf$, $\calq$ a finite partition of $\Omega$ with 
$\calq\subseteq\Sigma$ and $f\in L_p(\bo,\calf,\pp)$ with $\|f-\ave(f \, | \, \cala_{\calq})\|_{\Sigma}> \delta$. 
Then there exists a refinement $\calr$ of $\calq$ with $\calr\subseteq \Sigma$ and $|\calr|\mik |\calq| (k+1)$, and 
such that $\|\ave(f \, | \, \cala_{\calr})-\ave(f \, | \, \cala_{\calq})\|_{L_p} >\delta$.
\end{lem}
\begin{proof}
By our assumptions, there exists $S\in\Sigma$ such that
\begin{equation} \label{e3.9}
\big| \int_S \big( f-\ave(f \, | \, \cala_{\calq}) \big)\, d\pp \big|> \delta.
\end{equation}
Since $\Sigma$ is a $k$-semiring on $\Omega$, there exists a refinement $\calr$ of $\calq$ such that: (i) $\calr\subseteq\Sigma$,
(ii) $|\calr|\mik |\calq| (k+1)$, and (iii) $S\in\cala_{\calr}$. It follows, in particular, that
\begin{equation} \label{e3.10}
\int_S \ave(f \, | \, \cala_{\calr})\, d\pp=\int_S f\, d\pp.
\end{equation}
Hence, by \eqref{e3.9} and the monotonicity of the $L_p$ norms, we obtain that
\begin{eqnarray} \label{e3.11}
\delta & < & \big| \int_S \big( \ave(f \, | \, \cala_{\calr})-\ave(f \, | \, \cala_{\calq}) \big)\, d\pp \big| \\
& \mik & \|\ave(f \, | \, \cala_{\calr})-\ave(f \, | \, \cala_{\calq})\|_{L_1} 
\mik \|\ave(f \, | \, \cala_{\calr})-\ave(f \, | \, \cala_{\calq})\|_{L_p} \nonumber
\end{eqnarray}
and the proof is completed.
\end{proof}
We proceed with the following lemma.
\begin{lem} \label{l3.4}
Let $k, \ell$ be positive integers, $0<\delta,\sigma \mik 1$ and $1<p\mik 2$, and set
\begin{equation} \label{e3.12}
n=\Big\lceil \frac{\sigma^2 \ell}{\delta^2 (p-1)}\Big\rceil.
\end{equation}
Also let $(\bo,\calf,\pp)$ be a probability space and let $(\Sigma_i)$ be an increasing sequence of $k\text{-semirings}$
on $\Omega$ with $\Sigma_i\subseteq \calf$ for every $i\in\nn$. Finally, let $m\in\nn$ and $\calp$ a partition of $\bo$ 
with $\calp\subseteq\Sigma_m$ and $|\calp|\mik (k+1)^m$. Then for every family $\calc$ in $L_p(\bo,\calf,\pp)$ with $|\calc|=\ell$
there exist $j\in\{m,\dots,m+n\}$ and a refinement $\calq$ of $\calp$ with $\calq\subseteq\Sigma_j$ and $|\calq|\mik (k+1)^j$,
and such that either
\begin{enumerate}
\item[(a)] $\|\ave(f \, | \, \cala_{\calq})-\ave(f \, | \, \cala_{\calp})\|_{L_p}>\sigma$ for some $f\in\calc$, or
\item[(b)] $\|\ave(f \, | \, \cala_{\calq})-\ave(f \, | \, \cala_{\calp})\|_{L_p}\mik \sigma$ and
$\|f-\ave(f \, | \, \cala_{\calq})\|_{\Sigma_{j+1}}\mik \delta$ for every $f\in\calc$.
\end{enumerate}
\end{lem}
The case ``$p=2$" in Lemma \ref{l3.4} can be proved with an ``energy increment strategy" which ultimately depends upon the fact
that martingale difference sequences are orthogonal in $L_2$ (see, e.g., \cite[Theorem 2.11]{Tao2}). In the non-Hilbertian 
case (that is, when $1<p<2$) the geometry is more subtle and we will rely, instead, on Proposition \ref{pa.1}. The argument
can therefore be seen as the $L_p$-version of the ``energy increment strategy"\!. More applications of this method
are given in \cite{DKK,DKT3}.
\begin{proof}[Proof of Lemma \emph{\ref{l3.4}}]
Assume that the first part of the lemma is not satisfied. Note that this is equivalent to saying that
\begin{enumerate}
\item[(H1)] for every $j\in\{m,\dots,m+n\}$, every refinement $\calq$ of $\calp$ with $\calq\subseteq\Sigma_j$ and
$|\calq|\mik (k+1)^j$ and every $f\in\calc$ we have $\|\ave(f\, | \, \cala_{\calq})-\ave(f\, | \, \cala_{\calp})\|_{L_p}\mik\sigma$.
\end{enumerate}
We will use hypothesis (H1) to show that part (b) is satisfied. 

To this end we will argue by contradiction. Let $j\in\{m,\dots,m+n\}$ and let $\calq$ be a refinement of $\calp$ with 
$\calq\subseteq\Sigma_j$ and $|\calq|\mik (k+1)^j$. Observe that hypothesis (H1) and our assumption that part (b) does 
not hold true, imply that there exists $f\in\calc$ (possibly depending on the partition $\calq$) such that 
$\|f-\ave(f \, | \, \cala_{\calq})\|_{\Sigma_{j+1}}>\delta$. Since the sequence $(\Sigma_i)$ is increasing, Lemma \ref{l3.3}
can be applied to the $k$-semiring $\Sigma_{j+1}$, the partition $\calq$ and the random variable $f$. Hence, we obtain that 
\begin{enumerate}
\item[(H2)] for every $j\in\{m,\dots,m+n\}$ and every refinement $\calq$ of $\calp$ with $\calq\subseteq\Sigma_j$ and
$|\calq|\mik (k+1)^j$ there exist  $f\in\calc$ and a refinement $\calr$ of $\calq$ with $\calr\subseteq\Sigma_{j+1}$ and
$|\calr|\mik (k+1)^{j+1}$, and such that $\|\ave(f \, | \, \cala_{\calr})-\ave(f \, | \, \cala_{\calq})\|_{L_p}> \delta$.
\end{enumerate}
Recursively and using hypothesis (H2), we select a finite sequence $\calp_0,\dots,\calp_n$ of partitions of $\Omega$
with $\calp_0=\calp$ and a finite sequence $f_1,\dots,f_n$ in $\calc$ such that for every $i\in [n]$ we have:
(P1) $\calp_i$ is a refinement of $\calp_{i-1}$, (P2) $\calp_i\subseteq \Sigma_{m+i}$ and $|\calp_i|\mik (k+1)^{m+i}$,
and (P3) $\|\ave(f_i \, | \, \cala_{\calp_i})-\ave(f_i \, | \, \cala_{\calp_{i-1}})\|_{L_p}>\delta$. It follows, 
in particular, that $(\cala_{\calp_i})_{i=0}^n$ is an increasing sequence of finite sub-$\sigma$-algebras of $\calf$. 
Also note that, by the classical pigeonhole principle and the fact that $|\calc|=\ell$, there exist $g\in\calc$ and 
$I\subseteq [n]$ with $|I|\meg n/\ell$ and such that $g=f_i$ for every $i\in I$. 

Next, set $f=g-\ave(g\, | \, \cala_{\calp})$ and let $(d_i)_{i=0}^n$ be the difference sequence associated with the finite 
martingale $\ave(f\, | \, \cala_{\calp_0}),\dots,\ave(f\, | \, \cala_{\calp_n})$. Observe that for every $i\in I$ we have
$d_i=\ave(g \, | \, \cala_{\calp_i})-\ave(g \, | \, \cala_{\calp_{i-1}})$ and so, by the choice of $I$ and property (P3), 
we obtain that $\|d_i\|_{L_p}>\delta$ for every $i\in I$. Therefore, by Proposition \ref{pa.1}, we have
\begin{eqnarray} \label{e3.13}
\sigma & \stackrel{\eqref{e3.12}}{\mik} & \sqrt{p-1}\,\delta \Big(\frac{n}{\ell}\Big)^{1/2} \mik
\sqrt{p-1}\, \delta |I|^{1/2} \\
& < & \sqrt{p-1} \cdot \Big( \sum_{i=0}^n \|d_i\|^2_{L_p} \Big)^{1/2} \nonumber \\ 
& \stackrel{\eqref{ea.5}}{\mik} & \big\| \sum_{i=0}^n d_i \big\|_{L_p} =
\|\ave(g \, | \, \cala_{\calp_n})-\ave(g \, | \, \cala_{\calp})\|_{L_p}. \nonumber 
\end{eqnarray}
On the other hand, by properties (P1) and (P2), we see that $\calp_n$ is a refinement of $\calp$ with $\calp_n\subseteq\Sigma_{m+n}$
and $|\calp_n|\mik (k+1)^{m+n}$. Therefore, by hypothesis (H1), we must have 
$\|\ave(g \, | \, \cala_{\calp_n}\!)-\ave(g \, | \, \cala_{\calp})\|_{L_p}\mik \sigma$ which contradicts, of course, the estimate
in \eqref{e3.13}. The proof of Lemma \ref{l3.4} is thus completed.
\end{proof}
The following lemma is the last step of the proof of Theorem \ref{t3.1}.
\begin{lem} \label{l3.5}
Let $k, \ell$ be positive integers, $0<\sigma \mik 1$, $1<p\mik 2$ and $H\colon\nn\to\rr$ a growth function.
Set $L=\lceil \ell\, \sigma^{-2}(p-1)^{-1}\rceil$ and define $(n_i)$ recursively by the rule
\begin{equation} \label{e3.14}
\begin{cases}
n_0=0, \\
n_{i+1}=n_i+\lceil \sigma^2\, \ell\, H(n_i)^2 (p-1)^{-1}\rceil.
\end{cases}
\end{equation}
Also let $(\bo,\calf,\pp)$ be a probability space and let $(\Sigma_i)$ be an increasing sequence of $k$-semirings on $\Omega$
with $\Sigma_i\subseteq\calf$ for every $i\in\nn$. Finally, let $\calc$ be a family in $L_p(\bo,\calf,\pp)$ such that $\|f\|_{L_p}\mik 1$
for every $f\in\calc$ and with $|\calc|=\ell$. Then there exist $j\in \{0,\dots,L-1\}$, $J\in \{n_j,\dots,n_{j+1}\}$ and two partitions
$\calp, \calq$ of $\Omega$ with the following properties: \emph{(i)} $\calp\subseteq\Sigma_{n_j}$ and $\calq\subseteq \Sigma_J$, 
\emph{(ii)} $|\calp|\mik (k+1)^{n_j}$ and $|\calq|\mik (k+1)^J$, \emph{(iii)} $\calq$ is a refinement of\, $\calp$, and \emph{(iv)} 
$\|\ave(f \, | \, \cala_{\calq})-\ave(f \, | \, \cala_{\calp})\|_{L_p}\mik\sigma$ and
$\|f-\ave(f \, | \, \cala_{\calq})\|_{\Sigma_{J+1}}\mik 1/H(n_j)$ for every $f\in\calc$.
\end{lem}
\begin{proof}
It is similar to the proof of Lemma \ref{l3.4}. Indeed, assume, towards a contradiction, that the lemma is false. 
Recursively and using Lemma \ref{l3.4}, we select a finite sequence $J_0,\dots, J_L$ in $\nn$ with $J_0=0$, a finite 
sequence $\calp_0,\dots,\calp_L$ of partitions of $\bo$ with $\calp_0=\{\bo\}$ and a finite sequence $f_1,\dots,f_L$ 
in $\calc$ such that for every $i\in [L]$ we have that: (P1) $J_i\in \{n_{i-1},\dots,n_i\}$, (P2) the partition $\calp_i$ is 
a refinement of $\calp_{i-1}$, (P3) $\calp_i\subseteq \Sigma_{J_i}$ with $|\calp_i|\mik (k+1)^{J_i}$, and (P4)
$\|\ave(f_i \, | \, \cala_{\calp_i})-\ave(f_i \, | \, \cala_{\calp_{i-1}})\|_{L_p}>\sigma$. As in the proof of Lemma \ref{l3.4},
we observe that $(\cala_{\calp_i})_{i=0}^L$ is an increasing sequence of finite sub-$\sigma$-algebras of $\calf$,
and we select $g\in\calc$ and $I\subseteq [L]$ with $|I|\meg L/\ell$ and such that $g=f_i$ for every $i\in I$. 
Let $(d_i)_{i=0}^L$ be the difference sequence associated with the finite martingale 
$\ave(g\, | \, \cala_{\calp_0}),\dots,\ave(g\, | \, \cala_{\calp_L})$. Notice that, by property (P4), we have $\|d_i\|_{L_p}>\sigma$
for every $i\in I$. Hence, by the choice of $L$, Proposition \ref{pa.1} and the fact that $\|g\|_{L_p}\mik 1$, we conclude that
\begin{eqnarray} \label{e3.15}
1 & \mik & \sqrt{p-1}\, \sigma |I|^{1/2} < 
\sqrt{p-1} \cdot \Big( \sum_{i=0}^L \|d_i\|^2_{L_p} \Big)^{1/2} \\ 
& \stackrel{\eqref{ea.5}}{\mik} & \big\| \sum_{i=0}^L d_i \big\|_{L_p} =
\|\ave(g \, | \, \cala_{\calp_L})\|_{L_p} \mik \|g\|_{L_p}\mik 1 \nonumber 
\end{eqnarray}
which is clearly a contradiction. The proof of Lemma \ref{l3.5} is completed.
\end{proof}
We are ready to complete the proof of Theorem \ref{t3.1}.
\begin{proof}[Proof of Theorem \emph{\ref{t3.1}}]
Fix the data $k, \ell, \sigma, p$, the growth function $F$, the sequence $(\cals_i)$ and the family $\calc$. We define 
$H\colon\nn\to\rr$ by the rule $H(n)=F^{(n+2)}(0)$ and we observe that $H$ is a growth function. Moreover, for every 
$i\in\nn$ let $m_i=F^{(i)}(0)$ and set $\Sigma_i=\cals_{m_i}$. Notice that $(\Sigma_i)$ is an increasing sequence of 
$k$-semirings of $\bo$ with $\Sigma_i\subseteq\calf$ for every $i\in\nn$.

Let $j,J,\calp$ and $\calq$ be as in Lemma \ref{l3.5} when applied to $k,\ell,\sigma,p,H$, the sequence $(\Sigma_i)$ and the 
family $\calc$. We set 
\begin{equation} \label{e3.16}
N=m_{n_j}=F^{(n_j)}(0)
\end{equation}
and we claim that the natural number $N$ and the partitions $\calp$ and $\calq$ are as desired.

Indeed, notice first that $n_j\mik n_{L-1}$. Since $F$ is a growth function, by the choice of $h$ and $R$ in \eqref{e3.1} 
and \eqref{e3.2} respectively, we have
\begin{equation} \label{e3.17}
N\mik F^{(n_{L-1})}(0)=F^{(R)}(0) \stackrel{\eqref{e3.3}}{=} \mathrm{Reg}(k,\ell,\sigma,p,F).
\end{equation}
On the other hand, note that $n_j\mik F^{(n_j)}(0)=N$ and so $|\calp|\mik (k+1)^{n_j}\mik (k+1)^N$ and 
$\calp\subseteq \Sigma_{n_j}=\cals_N$. Moreover, by Lemma \ref{l3.5}, we see that $\calq$ is a finite refinement of $\calp$ 
with $\calq\subseteq \cals_i$ for some $i\meg N$. It follows that $N, \calp$ and $\calq$ satisfy the requirements of the theorem. 
Finally, let $f\in\calc$ be arbitrary and write $f=f_{\mathrm{str}}+ f_{\mathrm{err}}+ f_{\mathrm{unf}}$ where 
$f_{\mathrm{str}}=\ave(f \, | \, \cala_{\calp})$, $f_{\mathrm{err}}=\ave(f \, | \, \cala_{\calq})-\ave(f \, | \, \cala_{\calp})$
and $f_{\mathrm{unf}}=f-\ave(f \, | \, \cala_{\calq})$. Invoking Lemma \ref{l3.5}, we obtain that
\begin{equation} \label{e3.18}
\|f_{\mathrm{err}}\|_{L_p} = \|\ave(f \, | \, \cala_{\calq})-\ave(f \, | \, \cala_{\calp})\|_{L_p} \mik \sigma. 
\end{equation}
Also observe that $n_j+1\mik J+1$ which is easily seen to imply that $\cals_{F(N)}\subseteq \Sigma_{J+1}$. 
Therefore, using Lemma \ref{l3.5} once again, for every $i\in\{0,\dots, F(N)\}$ we have
\begin{eqnarray} \label{e3.19}
\|f_{\mathrm{unf}}\|_{\cals_i} & = & \|f-\ave(f \, | \, \cala_{\calq}) \|_{\cals_i} \mik
\|f-\ave(f \, | \, \cala_{\calq}) \|_{\Sigma_{J+1}} \\
& \mik & \frac{1}{H(n_j)} = \frac{1}{F\big( F(N)\big)} \mik \frac{1}{F(i)}. \nonumber
\end{eqnarray}
The proof of Theorem \ref{t3.1} is completed.
\end{proof}


\section{Applications}

\numberwithin{equation}{section}


\subsection{Uniform partitions}

In this section we will discuss some applications of our main result (more applications can be found in \cite{DK}).
We start with a consequence of Theorem \ref{t3.1} which is closer in spirit to the original formulation of Szemer\'{e}di's 
regularity lemma \cite{Sz}.

Recall that if $(\bo,\calf,\pp)$ is a probability space, $f\in L_1(\bo,\calf,\pp)$ and $S\in\calf$ is an event of 
non-zero probability, then $\ave(f\, | \, S)$ stands for the conditional expectation of $f$ with respect to $S$, that is,
$\ave(f\, | \, S) =\big( \int_S f\, d\pp\big)/ \pp(S)$. If $\pp(S)=0$, then by convention we set $\ave(f\, | \, S)=0$. 
We have the following definition.
\begin{defn} \label{d4.1}
Let  $(\bo,\calf,\pp)$ be a probability space, $k$ a positive integer and $\cals$ a $k$-semiring on $\bo$ with $\cals\subseteq\calf$.
Also let $f\in L_1(\bo,\calf,\pp)$, $0<\eta\mik 1$ and $S\in\cals$. We say that the set $S$ is \emph{$(f,\cals,\eta)$-uniform} if for
every $T\subseteq S$ with $T\in\cals$ we have
\begin{equation} \label{e4.1}
\big| \int_T \big(f-\ave(f\, | \, S)\big)\, d\pp \big| \mik \eta\cdot \pp(S).
\end{equation}
Moreover, for every $\calc\subseteq\cals$ we set $\mathrm{Unf}(\calc,f,\eta) = \{ C\in\calc: C \text{ is $(f,\cals,\eta)$-uniform}\}$.
\end{defn}
Notice that if $S\in\cals$ with $\pp(S)=0$, then the set $S$ is $(f,\cals,\eta)$-uniform for every $0<\eta\mik 1$. The same remark 
of course applies if the random variable $f$ is constant on $S$. Also note that the concept of $(f,\cals,\eta)$-uniformity is closely
related to the $\cals\text{-uniformity}$ norm. Indeed, let $S\in\cals$ with $\pp(S)>0$ and observe that the set $S$ is
$(f,\cals,\eta)$-uniform if and only if the function $f-\ave(f \, | \, S)$, viewed as a random variable in $L_1(\bo,\calf,\pp_S)$,
has $\cals$-uniformity norm less than or equal to $\eta$. (Here, $\pp_S$ stands for the conditional probability measure of $\pp$
relative to $S$.) In particular, the set $\bo$ is $(f,\cals,\eta)$-uniform if and only if $\|f-\ave(f)\|_{\cals}\mik\eta$.

We have the following proposition (see also \cite[Section 11.6]{TV}).
\begin{prop} \label{p4.2}
For every positive integer $k$, every $1<p\mik 2$ and every $0<\eta\mik 1$ there exists a positive integer $\mathrm{U}(k,p,\eta)$
with the following property. If $(\bo,\calf,\pp)$ is a probability space, $\cals$ a $k$-semiring on $\bo$ with $\cals\subseteq \calf$
and $f\in L_p(\bo,\calf,\pp)$ with $\|f\|_{L_p}\mik 1$, then there exist a positive integer $M\mik \mathrm{U}(k,p,\eta)$ and a partition
$\calp$ of $\bo$ with $\calp\subseteq\cals$ and $|\calp|=M$, and such that
\begin{equation} \label{e4.2}
\sum_{S\in\mathrm{Unf}(\calp,f,\eta)} \pp(S)\meg 1-\eta.
\end{equation}
\end{prop}
The following lemma will enable us to reduce Proposition \ref{p4.2} to Corollary \ref{c3.2}.
\begin{lem} \label{l4.3}
Let $(\bo,\calf,\pp)$ be a probability space, $k$ a positive integer and $\cals$ a $k\text{-semiring}$ on $\bo$ with
$\cals\subseteq \calf$. Also let $\calp$ be a finite partition of $\bo$ with $\calp\subseteq \calf$, $f\in L_1(\bo,\calf,\pp)$ 
and $0<\eta\mik 1$. Assume that the function $f$ admits a decomposition $f=f_{\mathrm{str}}+ f_{\mathrm{err}}+f_{\mathrm{unf}}$ 
into integrable random variables such that $f_{\mathrm{str}}$ is constant on each $S\in\calp$ and the functions $f_{\mathrm{err}}$
and $f_{\mathrm{unf}}$ obey the estimates $\|f_{\mathrm{err}}\|_{L_1}\mik \eta^2/8$ and 
$\|f_{\mathrm{unf}}\|_{\cals}\mik (\eta^2/8) |\calp|^{-1}$. Then we have
\begin{equation} \label{e4.3}
\sum_{S\notin\mathrm{Unf}(\calp,f,\eta)} \pp(S) \mik \eta.
\end{equation}
\end{lem}
\begin{proof}
Fix $S\notin \mathrm{Unf}(\calp,f,\eta)$. We select $T\subseteq S$ with $T\in\cals$ such that
\begin{equation} \label{e4.4}
\eta\cdot \pp(S) < \big| \int_T \big(f -\mathbb{E}(f\, | \, S)\big)\, d\pp \big|.
\end{equation}
The function $f_{\mathrm{str}}$ is constant on $S$ and so, by \eqref{e4.4}, we see that
\begin{equation} \label{e4.5}
\eta\cdot\pp(S) < \big|\int_T \big(f_{\mathrm{err}}-\mathbb{E}(f_{\mathrm{err}} \, | \, S)\big)\, d\pp \big| +
\big|\int_T\big(f_{\mathrm{unf}} -\mathbb{E}(f_{\mathrm{unf}} \, | \, S)\big)\, d\pp\big|.
\end{equation}
Next observe that
\begin{equation} \label{e4.6}
\big|\int_T\big(f_{\mathrm{err}}- \mathbb{E}(f_{\mathrm{err}} \, | \, S)\big)\, d\pp \big|\mik
2 \mathbb{E}( |f_{\mathrm{err}}| \, | \, S) \cdot \pp(S)
\end{equation}
and
\begin{equation} \label{e4.7}
\big| \int_T\big(f_{\mathrm{unf}} -\mathbb{E}(f_{\mathrm{unf}} \, | \, S)\big)\, d\pp\big| \mik 2\|f_{\mathrm{unf}}\|_{\cals}.
\end{equation}
Finally, notice that $\pp(S)>0$ since $S\notin\mathrm{Unf}(\calp,f,\eta)$. Thus, setting
\begin{equation} \label{e4.8}
\cala=\{S\in\calp: \mathbb{E}( |f_{\mathrm{err}}| \, | \, S)\meg \eta/4\} \ \text{ and } \
\calb=\{S\in\calp: \pp(S)\mik 4\eta^{-1}\|f_{\mathrm{unf}}\|_{\cals}\} 
\end{equation}
and invoking \eqref{e4.5}--\eqref{e4.7}, we obtain that $\calp\setminus \mathrm{Unf}(\calp,f,\eta)\subseteq \cala\cup\calb$.

Since the family $\calp$ is a partition, it consists of pairwise disjoint sets. Hence, 
\begin{equation} \label{e4.9}
\sum_{S\in\cala} \pp(S) \mik \frac{4}{\eta}\Big( \sum_{S\in\cala} \int_S |f_{\mathrm{err}}|\, d\pp \Big) \mik
\frac{4}{\eta} \|f_{\mathrm{err}}\|_{L_1} \mik \frac{\eta}{2}.
\end{equation}
Moreover,
\begin{equation} \label{e4.10}
\sum_{S\in\calb} \pp(S) \mik \frac{4\|f_{\mathrm{unf}}\|_{\cals}}{\eta} \cdot |\calb|\mik
\frac{4\|f_{\mathrm{unf}}\|_{\cals}}{\eta} \cdot |\calp| \mik \frac{\eta}{2}.
\end{equation}
By \eqref{e4.9} and \eqref{e4.10} and using the inclusion $\calp\setminus \mathrm{Unf}(\calp,f,\eta)\subseteq \cala\cup\calb$,
we conclude that the estimate in \eqref{e4.3} is satisfied and the proof is completed.
\end{proof}
We proceed to the proof of Proposition \ref{p4.2}.
\begin{proof}[Proof of Proposition \emph{\ref{p4.2}}]
Fix $k,p$ and $\eta$. We set $\sigma=\eta^2/8$ and we define $F\colon\nn\to \rr$ by the rule
$F(n)=(n/\sigma)+1=(8n/\eta^2)+1$ for every $n\in\nn$. Notice that $F$ is a growth function. We set
\begin{equation} \label{e4.11}
\mathrm{U}(k,p,\eta)= \mathrm{Reg}'(k,p,\sigma, F)
\end{equation}
and we claim that $\mathrm{U}(k,p,\eta)$ is as desired. Indeed, let $(\bo,\calf,\pp)$ be a probability
space and $\cals$ a $k$-semiring on $\Omega$ with $\cals\subseteq\calf$. Also let $f\in L_p(\bo,\calf,\pp)$ with 
$\|f\|_{L_p}\mik 1$. By Corollary \ref{c3.2}, there exist a positive integer $M\mik \mathrm{U}(k,p,\eta)$, a partition $\calp$ 
of $\bo$ with $\calp\subseteq\cals$ and $|\calp|=M$, and a finite refinement $\calq$ of $\calp$ with $\calq\subseteq\cals$
such that, setting 
\begin{equation} \label{e4.12}
f_{\mathrm{str}}=\ave(f \, | \, \cala_{\calp}), \ \ f_{\mathrm{err}}=\ave(f \, | \, \cala_{\calq})-\ave(f \, | \, \cala_{\calp}) 
\ \text{ and } \ f_{\mathrm{unf}}=f-\ave(f \, | \, \cala_{\calq}),
\end{equation}
we have the estimates $\|f_{\mathrm{err}}\|_{L_p}\mik \sigma$ and $\|f_{\mathrm{unf}}\|_{\cals}\mik 1/F(M)$. It follows
that $f$ admits a decomposition $f=f_{\mathrm{str}}+f_{\mathrm{err}}+ f_{\mathrm{unf}}$ into integrable random
variables such that $f_{\mathrm{str}}$ is constant  on each $S\in\calp$, $\|f_{\mathrm{err}}\|_{L_p}\mik \sigma$ and
$\|f_{\mathrm{unf}}\|_{\cals}\mik 1/F(M)$. Notice that, by the monotonicity of the $L_p$ norms, we have 
$\|f_{\mathrm{err}}\|_{L_1}\mik \sigma$. Hence, by Lemma \ref{l4.3} and the choice of $\sigma$ and $F$, we conclude 
that the estimate in \eqref{e4.2} is satisfied and the proof of Proposition \ref{p4.2} is completed.
\end{proof}
We close this subsection by presenting an application of Proposition \ref{p4.2} for subsets of hypercubes 
(see also \cite[Section 2.1.3]{Tao3}). Specifically, let $A$ be a finite alphabet with $|A|\meg 2$ and set $K=|A|(|A|-1)2^{-1}$. 
Also let $n$ be a positive integer. As in Example \ref{ex3}, we view $A^n$ as a discrete probability space equipped 
with the uniform probability measure which we shall denote by $\pp$. More generally, for every nonempty subset $S$ 
of $A^n$ by  $\pp_S$ we shall denote the uniform probability measure concentrated on $S$, that is, 
$\pp_S(X)=|X\cap S|/|S|$ for every $X\subseteq A^n$. Recall that $\cals(A^n)$ stands for the $K\text{-semiring}$
on $A^n$ consisting of all subsets $X$ of $A^n$ which are written as
\begin{equation} \label{e4.13}
X=\bigcap_{\{a,b\}\in{A\choose 2}} X_{\{a,b\}}
\end{equation}
where $X_{\{a,b\}}$ is $(a,b)$-insensitive for every $\{a,b\} \in {A\choose 2}$.

Now let $D$ be a subset of $A^n$, $0<\ee\mik 1$ and $S\in\cals(A^n)$ with $S\neq\emptyset$. Notice that the set $S$ is
$(\mathbf{1}_D,\cals(A^n),\ee^2)$-uniform if and only if for every nonempty $T\subseteq S$ with $T\in\cals(A^n)$ we have
\begin{equation} \label{e4.14}
|\pp_T(D)-\pp_S(D)| \cdot \pp(T) \mik \ee^2 \cdot\pp(S).
\end{equation}
In particular, if $S$ is nonempty and $(\mathbf{1}_D,\cals(A^n),\ee^2)$-uniform, then for every $T\subseteq S$ with 
$T\in\cals(A^n)$ and $|T|\meg\ee |S|$ we have $|\pp_T(D)-\pp_S(D)| \mik\ee$. Thus, by Proposition \ref{p4.2} and taking
into account these remarks, we obtain the following corollary.
\begin{cor} \label{c4.4}
For every integer $k\meg 2$ and every $0<\ee\mik 1$ there exists a positive integer $N(k,\ee)$ with the following property.
If $n$ is a positive integer, $A$ is an alphabet with $|A|=k$ and $D$ is a subset of $A^n$, then there exist a positive integer
$M\mik N(k,\ee)$, a partition $\calp$ of $A^n$ with $\calp\subseteq\cals(A^n)$ and $|\calp|=M$, and a subfamily 
$\calp'\subseteq \calp$ with $\pp(\cup\calp')\meg 1-\ee$ such that
\begin{equation} \label{e4.15}
|\pp_T(D)-\pp_S(D)|\mik\ee
\end{equation}
for every $S\in\calp'$ and every $T\subseteq S$ with $T\in\cals(A^n)$ and $|T|\meg\ee |S|$.
\end{cor}


\subsection{$L_p$ graphons}

Our last application is an extension of the, so-called, \textit{strong regularity lemma for $L_2$ graphons} (see, e.g., \cite{L,LS}).
To state this extension we need to introduce some terminology and notation related to graphons. 

Let $(\bo,\calf,\pp)$ be a probability space and recall that a \textit{graphon}\footnote[1]{In several places 
in the literature, graphons are required to be $[0,1]$-valued, and the term \textit{kernel} is used for (not necessarily bounded)
integrable, symmetric random variables.} is an integrable random variable $W\colon\bo\times\bo\to\rr$ which is symmetric, that is, 
$W(x,y)=W(y,x)$ for every $x,y\in \bo$. If $p>1$ and $W$ is graphon which belongs to $L_p$, then $W$ is said to be an 
\textit{$L_p$ graphon} (see, e.g., \cite{BCCZ}). 

Now let $\calr$ be a finite partition of $\bo$ with $\calr\subseteq\calf$ and notice that the family
\begin{equation} \label{e4.16}
\calr^2=\{S\times T: S,T\in\calr\}
\end{equation}
is a finite partition of $\bo\times\bo$. As in Section 3, let $\cala_{\calr^2}$ be the $\sigma$-algebra on $\bo\times\bo$
generated by $\calr^2$ and observe that $\cala_{\calr^2}$ consists of measurable sets. If $W\colon\bo\times\bo\to\rr$ is
a graphon, then the conditional expectation of $W$ with respect to $\cala_{\calr^2}$ is usually denoted by $W_\calr$. 
Note that $W_{\calr}$ is also a graphon and satisfies (see, e.g., \cite{L})
\begin{equation} \label{e4.17}
\|W_{\calr}\|_{\square}\mik \|W\|_{\square}
\end{equation}
where $\|\cdot\|_{\square}$ is the cut norm defined in \eqref{e2.14}. On the other hand, by standard properties of the
conditional expectation (see, e.g., \cite{Du}), we have $\|W_{\calr}\|_{L_p}\mik \|W\|_{L_p}$ for any $p\meg 1$. It follows,
in particular, that $W_{\calr}$ is an $L_p$ graphon provided, of course, that $W\in L_p$.

We have the following corollary. 
\begin{cor}[Strong regularity lemma for $L_p$ graphons] \label{c4.5}
For every $0<\ee\mik 1$, every $1<p\mik 2$ and every positive function $h\colon\nn\to\rr$ there exists a positive
integer $\mathrm{s}(\ee,p,h)$ with the following property. If $(\bo,\calf,\pp)$ is a  probability space
and $W\colon\bo\times\bo\to\rr$ is an $L_p$ graphon with $\|W\|_{L_p}\mik 1$, then there exist a partition $\calr$ of $\bo$
with $\calr\subseteq\calf$ and $|\calr|\mik \mathrm{s}(\ee,p,h)$, and an $L_p$ graphon $U\colon\bo\times\bo\to\rr$ 
such that $\|W-U\|_{L_p}\mik \ee$ and $\|U-U_{\calr}\|_{\square}\mik h\big(|\calr|\big)$.
\end{cor}
\begin{proof}
Fix the constants $\ee, p$ and the function $h$, and define $F\colon\nn\to\rr$ by the rule
\begin{equation} \label{e4.18}
F(n)=(n+1)+\sum_{i=0}^n \frac{8}{h(i)}.
\end{equation}
Notice that $F$ is a growth function. We set
\begin{equation} \label{e4.19}
\mathrm{s}(\ee,p,h)=\mathrm{Reg}'(4,\ee,p,F)
\end{equation}
and we claim that with this choice the result follows.

Indeed, let $(\bo,\calf,\pp)$ be a probability space and fix an $L_p$ graphon $W\colon\bo\times\bo\to\rr$ with 
$\|W\|_{L_p}\mik 1$. Also let $\Sigma_{\square}$ be the $4$-semiring on $\bo\times\bo$ which is defined via formula \eqref{e2.15}
for the given probability space $(\bo,\calf,\pp)$. We apply Corollary \ref{c3.2} to $\Sigma_{\square}$ and the random variable
$W$ and we obtain
\begin{enumerate}
\item[(a)] a partition $\calp$ of $\bo\times\bo$ with $\calp\subseteq\Sigma_{\square}$ and $|\calp|\mik\mathrm{Reg}'(4,\ee,p,F)$, and 
\item[(b)] a finite refinement $\calq$ of $\calp$ with $\calq\subseteq\Sigma_{\square}$ 
\end{enumerate}
such that, writing the graphon $W$ as $W_{\mathrm{str}}+W_{\mathrm{err}}+W_{\mathrm{str}}$ where 
$W_{\mathrm{str}}=\ave(W\, |\, \cala_{\calp})$, $W_{\mathrm{err}}=\ave(W \, | \, \cala_{\calq})-\ave(W \, | \, \cala_{\calp})$
and $W_{\mathrm{unf}}=W-\ave(W \, | \, \cala_{\calq})$, we have the estimates $\|W_{\mathrm{err}}\|_{L_p}\mik \ee$ and 
$\|W_{\mathrm{unf}}\|_{\Sigma_{\square}}\mik 1/F\big(|\calp|\big)$. Note that, by (a) and (b) and the definition of the $4$-semiring 
$\Sigma_{\square}$ in \eqref{e2.15}, there exist two finite partitions $\calr,\calz$ of $\bo$ with $\calr,\calz\subseteq\calf$ and 
such that $\calp=\calr^2$ and $\calq=\calz^2$. It follows, in particular, that the random variables $W_{\mathrm{str}}, W_{\mathrm{err}}$
and $W_{\mathrm{unf}}$ are all $L_p$ graphons.

We will show that the partition $\calr$ and the $L_p$ graphon $U\coloneqq W_{\mathrm{str}}+W_{\mathrm{unf}}$ are as desired. 
To this end notice first that
\begin{equation} \label{e4.20}
|\calr| \mik |\calr^2| = |\calp| \mik \mathrm{Reg}'(4,\ee,p,F)\stackrel{\eqref{e4.19}}{=}\mathrm{s}(\ee,p,h).
\end{equation}
Next observe that 
\begin{equation} \label{e4.21}
\|W-U\|_{L_p} = \| W_{\mathrm{err}}\|_{L_p} \mik \ee.
\end{equation}
Finally note that, by \eqref{e4.17}, we have $\|(W_{\mathrm{unf}})_{\calr}\|_{\square}\mik \|W_{\mathrm{unf}}\|_{\square}$.
Moreover, the fact that $\calp=\calr^2$ and the choice of $W_{\mathrm{str}}$ yield that $(W_{\mathrm{str}})_{\calr}=W_{\mathrm{str}}$.
Therefore,
\begin{eqnarray} \label{e4.22}
\|U-U_{\calr}\|_{\square} & \mik & 2\|W_{\mathrm{unf}}\|_{\square} \stackrel{\eqref{e2.16}}{\mik} 
8\|W_{\mathrm{unf}}\|_{\Sigma_{\square}} \mik \frac{8}{F\big(|\calp|\big)} \\
& \stackrel{\eqref{e4.20}}{\mik} & \frac{8}{F\big(|\calr|\big)} \stackrel{\eqref{e4.18}}{\mik} h\big( |\calr|\big) \nonumber
\end{eqnarray}
and the proof of Corollary \ref{c4.5} is completed. 
\end{proof}
\begin{rem} \label{r1}
Recently, Borgs, Chayes, Cohn and Zhao \cite{BCCZ} extended the weak regularity lemma to $L_p$ graphons
for any $p>1$. Their extension follows, of course, from Corollary \ref{c4.5}, but this reduction is rather ineffective 
since the bound obtained by Corollary \ref{c4.5} is quite poor. However, this estimate can be significantly improved
if instead of invoking Corollary \ref{c3.2}, one argues directly as in the proof of Lemma \ref{l3.4}. More precisely, 
note that for every $0<\ee\mik 1$, every $1< p\mik 2$, every probability space $(\bo,\calf,\pp)$ and every
$L_p$ graphon $W\colon \bo\times\bo\to\rr$ with $\|W\|_{L_p}\mik 1$ there exists a partition $\calr$ of $\bo$
with $\calr\subseteq \calf$ and 
\begin{equation} \label{e4.23}
|\calr|\mik 4^{(p-1)^{-1}\ee^{-2}}
\end{equation}
and such that $\|W-W_{\calr}\|_{\square}\mik \ee$. The estimate in \eqref{e4.23} matches the bound for the weak regularity lemma
for the case of $L_2$ graphons (see, e.g., \cite{L}) and is essentially optimal.
\end{rem}


\appendix

\section{Martingale difference sequences} 

\numberwithin{equation}{section}

\subsection*{A.1}

Recall that a finite sequence $(f_i)_{i=0}^n$ of real-valued random variables on a probability space $(\bo,\calf,\pp)$ 
is said to be a \emph{martingale} if there exists an increasing sequence $(\calf_i)_{i=0}^n$ of sub-$\sigma$-algebras
of $\calf$ such that: (i) $f_i\in L_1(\bo,\calf_i,\pp)$ for every $i\in\{0,\dots,n\}$, and 
(ii) $f_i=\ave(f_{i+1}\, | \, \calf_i)$ if $n\meg 1$ and $i\in\{0,\dots,n-1\}$.

A \emph{martingale difference sequence} is the sequence of successive differences of a martingale. Specifically, 
a finite sequence $(d_i)_{i=0}^n$ of random variables on $(\bo,\calf,\pp)$ is a martingale difference
sequence if there exists a martingale $(f_i)_{i=0}^n$ such that $d_0=f_0$ and
\begin{equation} \label{ea.1}
d_i=f_i-f_{i-1}
\end{equation}
if $n\meg 1$ and $i\in [n]$. Note that this is equivalent to saying that there exists an increasing sequence $(\calf_i)_{i=0}^n$
of sub-$\sigma$-algebras of $\calf$ such that: (i) $d_i\in L_1(\bo,\calf_i,\pp)$ for every $i\in\{0,\dots,n\}$, and 
(ii) $\ave(d_i\, | \, \calf_{i-1})=0$ if $n\meg 1$ and $i\in [n]$.

\subsection*{A.2}

It is easy to see that martingale difference sequences are monotone basic sequences in $L_p$ for any $p\meg 1$; that is,
if $(d_i)_{i=0}^n$ is a martingale difference sequence in $L_p$ for some $p\meg 1$, then for every $0\mik k\mik n$ and
every $a_0,\dots,a_n\in\rr$ we have
\begin{equation} \label{ea.2}
\big\| \sum_{i=0}^k a_i d_i\big\|_{L_p} \mik \big\| \sum_{i=0}^n a_i d_i \big\|_{L_p}.
\end{equation}
It follows, in particular, that 
\begin{equation} \label{ea.3}
\big\| \sum_{i=k}^\ell d_i\big\|_{L_p} \mik 2\big\| \sum_{i=0}^n d_i \big\|_{L_p}
\end{equation}
for every $0\mik k\mik \ell\mik n$. Another basic property of martingale difference sequences is that they are orthogonal in $L_2$.
Therefore, for every martingale difference sequence $(d_i)_{i=0}^n$ in $L_2$ we have
\begin{equation} \label{ea.4}
\Big( \sum_{i=0}^n \|d_i\|_{L_2}^2 \Big)^{1/2} = \big\| \sum_{i=0}^n d_i \big\|_{L_2}.
\end{equation}
We will need the following extension of this fact.
\begin{prop} \label{pa.1}
Let $(\bo,\calf,\pp)$ be a probability space and $1<p \mik 2$. Then for every martingale difference 
sequence $(d_i)_{i=0}^n$ in $L_p(\Omega,\calf,\mathbb{P})$ we have
\begin{equation} \label{ea.5}
\Big( \sum_{i=0}^n \|d_i\|^2_{L_p} \Big)^{1/2} \mik \Big(\frac{1}{p-1}\Big)^{1/2} \cdot \big\| \sum_{i=0}^n d_i\big\|_{L_p}.
\end{equation}
\end{prop}

\subsection*{A.3}

Proposition \ref{pa.1} follows by iterating the following martingale convexity inequality 
which is due to Ricard and Xu \cite{RX}.
\begin{prop} \label{pa.2}
Let $(\bo,\calf,\pp)$ be a probability space and $1<p \mik 2$. Then for every sub-$\sigma$-algebra $\calb$ of\, $\calf$ and every
$f\in L_p(\Omega,\calf,\pp)$ we have
\begin{equation} \label{ea.6}
\|\ave(f\, |\, \calb)\|^2_{L_p} + (p-1) \|f-\ave(f\, |\, \calb)\|^2_{L_p} \mik \|f\|^2_{L_p}. 
\end{equation}
\end{prop}
A remarkable feature of Proposition \ref{pa.2} is the fact that the constant $(p-1)$ in \eqref{ea.6} is best possible.
A basic ingredient of its proof is the following uniform convexity inequality for $L_p$ spaces which first appeared in 
the work of Ball, Carlen and Lieb \cite{BCL} (see also \cite[Lemma 4.32]{Pi}).
\begin{prop} \label{p4}
Let $(\Omega,\Sigma,\mu)$ be an arbitrary measure space and $1<p \mik 2$. Then for every $x,y\in L_p(\Omega,\Sigma,\mu)$ we have
\begin{equation} \label{ea.7}
\|x\|_{L_p}^2 + (p-1) \|y\|_{L_p}^2 \mik \frac{\|x+y\|_{L_p}^2+\|x-y\|_{L_p}^2}{2}.
\end{equation}
\end{prop}
The deduction of the martingale inequality \eqref{ea.6} from \eqref{ea.7} is done via an elegant pseudo-differentiation argument
which we will briefly describe for the convenience of the reader.

Let $I$ be an open interval of $\rr$ and let $\varphi\colon I\to \rr$ be a function. Also let $t\in I$ and recall that the 
\emph{pseudo-derivative of second order of $\varphi$ at $t$} is defined by 
\[ D^2\varphi(t)\coloneqq \liminf_{h\to 0^+} \frac{\varphi(t+h)+\varphi(t-h)-2\varphi(t)}{h^2}.\] 
Observe that if $\varphi$ is twice differentiable at $t$, then $\varphi''(t)=D^2\varphi(t)$. Also note that if $D^2\varphi(t)\meg 0$ 
for every $t\in I$, then $\varphi$ is convex.

Now let $(\bo,\calf,\pp)$, $p$, $\calb$ and $f$ be as in Proposition \ref{pa.2}. We set $a=\ave(f\, |\, \calb)$ and
$b=f-\ave(f\, | \, \calb)$, and we define $\varphi\colon \rr\to\rr$ by 
\begin{equation} \label{ea.8}
\varphi(t)= \|a+tb\|_{L_p}^2 - (p-1)\, t^2 \|b\|_{L_p}^2. 
\end{equation}
Using \eqref{ea.7}, it is easy to see that $D^2\varphi(t)\meg 0$ for every $t\in\rr$ and, consequently, the function $\varphi$ is convex.
Next observe that the function $\psi\colon \rr\to\rr$ defined by $\psi(t)=\|a+tb\|^2_{L_p}$ is also convex. Moreover, we have
\[ \|a+tb\|_{L_p} \meg \|\ave(a+tb\, |\,\calb )\|_{L_p} = \|a\|_{L_p} \]
which yields that the right derivative of $\psi$ at $0$ is positive. Noticing that the right derivative of $\varphi$ at $0$ coincides 
with that of $\psi$, we conclude that $\varphi$ must be increasing on $[0,1]$ which, in turn, easily implies \eqref{ea.6}. For more details, 
as well as noncommutative extensions, we refer to \cite{RX}. 


\section{ }

\numberwithin{equation}{section}

Our goal in this appendix is to use Theorem \ref{t3.1} to show the well-known fact that, for any $1<p\mik 2$, every $L_p$ bounded martingale 
is $L_p$ convergent  (see, e.g.,~\cite{Du}). Besides its intrinsic interest, this result also implies that Theorem \ref{t3.1} does not hold
true for the end-point case $p=1$. In fact, based on the argument below, one can easily construct a counterexample to Theorem \ref{t3.1}
using any $L_1$ bounded martingale which is not $L_1$ convergent.

We will need the following known approximation result (see, e.g., \cite{Pi}). We recall the proof for the convenience of the reader.
\begin{lem} \label{lb.1}
Let $(\bo,\calf,\pp)$ be a probability space and $p\meg 1$. Also let $(g_i)$ be a martingale in $L_p(\bo,\calf,\pp)$ and $\delta>0$.
Then there exist an increasing sequence $(\calf_i)$ of finite $\text{sub-}\sigma\text{-algebras}$ of $\calf$ and a martingale $(f_i)$ 
adapted to the filtration $(\calf_i)$ such that $\|g_i-f_i\|_{L_p}\mik \delta$ for every $i\in\nn$.
\end{lem}
\begin{proof}
Fix a filtration $(\calb_i)$ such that $(g_i)$ is adapted to $(\calb_i)$ and let $(\Delta_i)$ be the martingale difference
sequence associated with $(g_i)$. Recursively and using the fact that the set of simple functions is dense in $L_p$,
we select an increasing sequence $(\calf_i)$ of finite sub-$\sigma$-algebras of $\calf$ and a sequence $(s_i)$ of simple
functions such that for every $i\in\nn$ we have that: (i) $\calf_i$ is contained in $\calb_i$, 
(ii) $\|\Delta_i-s_i\|_{L_p}\mik \delta/2^{i+2}$, and (iii) $s_i\in L_p(\bo,\calf_i,\pp)$. For every $i\in\nn$ let 
$d_i=\ave(\Delta_i\, | \, \calf_i)$ and notice that the sequence $(d_i)$ is a martingale difference sequence since, by (i),
\begin{eqnarray} \label{eb.1}
\ave(d_{i+1}\, | \, \calf_i) & = & \ave\big(\ave(\Delta_{i+1}\, | \, \calf_{i+1})\, | \, \calf_i\big) \\
& = & \ave(\Delta_{i+1}\, | \, \calf_i) = \ave\big(\ave(\Delta_{i+1}\, | \, \calb_i)\, | \, \calf_i\big) =0. \nonumber
\end{eqnarray}
Thus, setting $f_i=d_0+\dots+ d_i$, we see that $(f_i)$ is a martingale adapted to the filtration $(\calf_i)$. Moreover, 
by (ii) and (iii), for every $i\in\nn$ we have
\begin{eqnarray} \label{eb.2}
\ \ \ \ \ \ \|g_i-f_i\|_{L_p} \!\! & \mik & \sum_{k=0}^i \|\Delta_k-d_k\|_{L_p} \mik \frac{\delta}{2} + \sum_{k=0}^i \|s_k-d_k\|_{L_p} \\
& = & \frac{\delta}{2} + \sum_{k=0}^i \|\ave(s_k-\Delta_k\, | \, \calf_k)\|_{L_p} 
\mik \frac{\delta}{2} + \sum_{k=0}^i \|s_k-\Delta_k\|_{L_p} \mik \delta \nonumber
\end{eqnarray}
and the proof is completed. 
\end{proof}
Now fix $1<p\mik 2$ and a probability space $(\bo,\calf,\pp)$, and assume, towards a contradiction, 
that there exists a bounded martingale $(g_i)$ in $L_p(\bo,\calf,\pp)$ which is not norm convergent. By \eqref{ea.3}, we see 
that $(g_i)$ has no convergent subsequence whatsoever. Therefore, by passing to a subsequence of $(g_i)$ and rescaling, 
we may assume that there exists $0<\ee\mik 1/3$ such that: (i) $\|g_i\|_{L_p}\mik 1/2$ for every $i\in\nn$, and (ii)
$\|g_i-g_j\|_{L_p}\meg 3\ee$ for every $i,j\in\nn$ with $i\neq j$. By Lemma \ref{lb.1} applied to the martingale 
$(g_i)$ and the constant ``$\delta=\ee$", there exist
\begin{enumerate}
\item[(P1)] an increasing sequence $(\calf_i)$ of finite $\text{sub-}\sigma\text{-algebras}$ of $\calf$, and
\item[(P2)] a martingale $(f_i)$ adapted to the filtration $(\calf_i)$
\end{enumerate}
such that $\|g_i-f_i\|_{L_p}\mik \ee$ for every $i\in\nn$. Hence,
\begin{enumerate}
\item[(P3)] $\|f_i\|_{L_p}\mik 1$ for every $i\in\nn$, and
\item[(P4)] $\|f_i-f_j\|_{L_p}\meg \ee$ for every $i,j\in\nn$ with $i\neq j$. 
\end{enumerate}
Notice that, by (P1), for every $i\in\nn$ the space $L_p(\bo,\calf_i,\pp)$ is finite-dimensional.
Since $\|\cdot\|_{\calf_i}$ is a norm on $L_p(\bo,\calf_i,\pp)$, there exists a constant $C_i\meg 1$ such that 
\begin{equation} \label{eb.3}
\|f\|_{\calf_i} \mik \|f\|_{L_p} \mik C_i \|f\|_{\calf_i}
\end{equation}
for every $f\in L_p(\bo,\calf_i,\pp)$.

Define $F\colon \nn\to\rr$ by the rule
\begin{equation} \label{eb.4}
F(i)= (i+1) + (8/\ee) \sum_{j=0}^i C_i
\end{equation}
and observe that $F$ is a growth function. Next, set
\begin{equation} \label{eb.5}
n=F\big(\mathrm{Reg}(1,1,\ee/8,p,F)\big)+1
\end{equation}
and let $(\cals_i)$ be defined by $\cals_i=\calf_i$
if $i\mik n$ and $\cals_i=\calf_n$ if $i>n$. Clearly, $(\cals_i)$ is an increasing sequence of $1$-semirings on $\bo$. We apply
Theorem \ref{t3.1} to the probability space $(\bo,\calf_n,\pp)$, the sequence $(\cals_i)$ and the random variable $f_n$, and we obtain
a natural number $N\mik \mathrm{Reg}(1,1,\ee/8,p,F)$, a finite partition $\calp$ of $\bo$ with $\calp\subseteq \cals_N$ and a finite
refinement $\calq$ of $\calp$ such that, writing $f_n=f_{\mathrm{str}}+f_{\mathrm{err}}+f_{\mathrm{unf}}$ where
\[ f_{\mathrm{str}}=\ave(f_n \, | \, \cala_{\calp}), \ \ f_{\mathrm{err}}=\ave(f_n \, | \, \cala_{\calq})-\ave(f_n \, | \, \cala_{\calp})
\ \text{ and } \ f_{\mathrm{unf}}=f_n-\ave(f_n \, | \, \cala_{\calq}),\]
we have that $\|f_{\mathrm{err}}\|_{L_p}\mik \ee/8$ and 
$\|f_{\mathrm{unf}}\|_{\cals_i}\mik 1/F(i)$ for every $i\in\{0,\dots,F(N)\}$. In particular, by the choice of $n$
and $(\cals_i)$, we see that 
\begin{equation} \label{eb.6}
\|f_{\mathrm{err}}\|_{L_p}\mik \frac{\ee}{8} \ \text{ and } \ \|f_{\mathrm{unf}}\|_{\calf_{N+1}}\mik \frac{1}{F(N+1)}.
\end{equation}

Now observe that, by property (P2),
\begin{equation} \label{eb.7}
f_N=\ave(f_n\, | \,\calf_N)= \ave(f_{\mathrm{str}}\, | \, \calf_N)+\ave(f_{\mathrm{err}}\, | \, \calf_N) +
\ave(f_{\mathrm{unf}}\, | \, \calf_N)
\end{equation}
and, similarly, 
\begin{equation} \label{eb.8}
f_{N+1}=\ave(f_n\, | \,\calf_{N+1})= \ave(f_{\mathrm{str}}\, | \, \calf_{N+1})+\ave(f_{\mathrm{err}}\, | \, \calf_{N+1})
+ \ave(f_{\mathrm{unf}}\, | \, \calf_{N+1}).
\end{equation}
The fact that $\calp\subseteq\cals_N$ yields that $\cala_{\calp}\subseteq \calf_N\subseteq \calf_{N+1}$ and so 
\begin{equation} \label{eb.9}
f_{\mathrm{str}}=\ave(f_{\mathrm{str}}\, | \, \calf_N) = \ave(f_{\mathrm{str}}\, | \, \calf_{N+1}).
\end{equation}
On the other hand, by \eqref{eb.6}, we have
\begin{equation} \label{eb.10}
\|\ave(f_{\mathrm{err}}\, | \, \calf_N)\|_{L_p}\mik \frac{\ee}{8} \ \text{ and } \ 
\|\ave(f_{\mathrm{err}}\, | \, \calf_{N+1})\|_{L_p} \mik \frac{\ee}{8}. 
\end{equation}
Finally, notice that $\ave(f_{\mathrm{unf}}\, | \, \calf_N)\in L_p(\bo,\calf_N,\pp)$. Thus, by \eqref{eb.3} and
Lemma \ref{l2.5}, we obtain that
\begin{eqnarray} \label{eb.11}
\|\ave(f_{\mathrm{unf}}\, | \, \calf_N)\|_{L_p} & \mik & C_N \|\ave(f_{\mathrm{unf}}\, | \, \calf_N)\|_{\calf_N}
\mik C_N \|f_{\mathrm{unf}}\|_{\calf_N} \\
& \mik & C_N \|f_{\mathrm{unf}}\|_{\calf_{N+1}} \stackrel{\eqref{eb.6}}{\mik} \frac{C_N}{F(N+1)}
\stackrel{\eqref{eb.4}}{\mik} \frac{\ee}{8}. \nonumber
\end{eqnarray}
With identical arguments we see that 
\begin{equation} \label{eb.12}
\|\ave(f_{\mathrm{unf}}\, | \, \calf_{N+1})\|_{L_p} \mik \frac{\ee}{8}.
\end{equation}
Combining \eqref{eb.7}--\eqref{eb.12}, we conclude that $\|f_N-f_{N+1}\|_{L_p}\mik \ee/2$ which contradicts, of course, property (P4). 
Hence, every bounded martingale in $L_p(\bo,\calf,\pp)$ is norm convergent, as desired.


\end{document}